\documentclass[11pt,letterpaper]{amsart}
\usepackage{amsmath,amssymb}
\usepackage{mathrsfs}
\usepackage{caption}
\usepackage{amsmath}
\usepackage{amsthm}
\usepackage{epsfig}
\usepackage{graphicx}
\newtheorem{thm}{Theorem}[section]
\newtheorem{lem}[thm]{Lemma}
\newtheorem{cor}[thm]{Corollary}

\newtheorem{ex}{Example}

\title[The Sylvester Theorem and the Rogers-Ramanujan Identities]{The Sylvester Theorem and the Rogers-Ramanujan Identities over Totally Real Number Fields}
\author[S.J.]{Se Wook Jang}
\address{Gangneung-Wonju National University\\ Gnagneung 25457, Republic of Korea}
\email{dcgg@naver.com} 

\author[B.K.]{Byeong Moon Kim}
\address{Gangneung-Wonju National University\\ Gnagneung 25457, Republic of Korea}
\email{kbm@gwnu.ac.kr} 

\author[K.K.]{Kwang Hoon Kim}
\address{Gangneung-Wonju National University\\ Gnagneung 25457, Republic of Korea}
\email{h2078@naver.com} 

\thanks{This work was supported by the National Research Foundation of Korea(NRF) grant funded by the Korea government(MSIT) (No. RS-2023-00247457)
}

\begin{document}
\begin{abstract}
In this paper, we prove two identities on the partition of a totally positive algebraic integer over a totally real number field which are the generalization of the Sylvester Theorem and that of the Rogers-Ramanujan Identities. Additionally, we give an another version of generalized Rogers-Ramanujan Identities.
\end{abstract}

\maketitle

\section{Introduction}
Among the numerous beautiful identities in mathematics, one of the most outstanding ones is the Rogers-Ramanujan Identities. These are two identities \[\sum_{n=0}^{\infty}\frac{q^{n^2}}{(1-q)(1-q^2)\cdots(1-q^n)}=\prod_{n=1}^{\infty}(1-q^{5n-1})^{-1}(1-q^{5n-4})^{-1}\] and \[\sum_{n=0}^{\infty}\frac{q^{n^2+n}}{(1-q)(1-q^2)\cdots(1-q^n)}=\prod_{n=1}^{\infty}(1-q^{5n-2})^{-1}(1-q^{5n-3})^{-1},\] which mean that the number of partitions of $n$ with the difference between any two distinct parts by at least 2 is equal to that of partitions of $n$ into parts congruent to $1$ or $4$ modulo $5$ and that the number of partitions of $n$ with each part exceeds $1$ and the difference between any two distinct parts by at least 2 is equal to that of partitions of $n$ into parts congruent to $2$ or $3$ modulo $5$. This fascinating identities have inspired many mathematicians and have constantly motivated the development of mathematics especially in the field of partition theory. Many mathematicians have issued a challenge to generalize these astonishing identities and many researches on similar type partition identities have been actively conducted. The results of G$\rm\ddot{o}$llnitz \cite{gol} and Gordon \cite{gor,gor1} in 1960s are prominant achievements among them. They are \[\sum_{n=0}^{\infty}\frac{(-q;q^2)_{n}q^{n^2}}{(q^2;q^2)_n}=\frac1{(q;q^8)_{\infty}(q^4;q^8)_{\infty}(q^7;q^8)_{\infty}}\] and \[\sum_{n=0}^{\infty}\frac{(-q;q^2)_{n}q^{n^2+n}}{(q^2;q^2)_n}=\frac1{(q^3;q^8)_{\infty}(q^4;q^8)_{\infty}(q^5;q^8)_{\infty}}\] where $(a;q)_n=(1-a)(1-aq)(1-aq^2)\cdots(1-aq^{n-1})$ and $(a;q)_{\infty}=\displaystyle\lim_{n\rightarrow\infty}(a;q)_n$. In 2016,  Griffin, Ono and  Warnaar \cite{gow} proved that there are infinite number of the Rogers-Ramanujan type identities. Despite these various significant accomplishments, the generalizations of these identities to all totally real number fields have not been attempted. 

In 1950 Rademacher \cite{rad} initiated the study of partition theory over real quadratic fields. He obtained an upper bound of the number of partitions of a totally positive algebraic integer of real quadratic fields. Mitsui \cite{mit} improved this result. For a long time, no further results are known. Recently the authors of this paper \cite{jkk} introduced the formal power $q$-sum which is a natural generalization of $q$-series. As a consequence, they generalize the Euler-Glaisher Theorem to totally real number fields. In addition, it is proved that for a totally positive algebraic integer $\delta$ over a totally real number field $K$, the number of solutions to $x_1+2x_2+\cdots+nx_n=\delta$ with $x_i$ $0$ or totally positive is equal to that of chain partitions of $\delta$ with at most $n$ parts. A partition is a $\it{chain~partition}$ if the difference of any two parts is $0$ or totally positive. We can see that the field $\mathbb Q$ of rational numbers is the only totally real number field such that every partition is a chain partition. It follows that the number of solutions to $x_1+2x_2+\cdots+nx_n=\delta$ does not coincide with that of partitions of $\delta$ with at most $n$ parts except $K=\mathbb Q$. 

Among a great number of significant identities in the partition theory of natural numbers, the identities that are generalized naturally to all totally real number fields are hardly found except Euler-Glaisher Theorem. The aim of this paper is to generalize the Sylvester Theorem and the Rogers-Ramanujan Identities to all totally real number fields. The direct generalizations of the latter identities does not hold when $K\ne\mathbb Q$, and thus we introduce the radial decomposition of a totally positive integer. With this simple concept, a way to generalize the Rogers-Ramanujan Identities to all totally real number fields is found. This decomposition is naturally arisen to generalize consecutive integers in Sylvester's Theorem.

\section{Basic Definitions and Theorems}
\subsection{The Partitions of Positive Algebraic Integers}
Let $K$ be a totally real number field and $\mathcal O^+=\mathcal O_K^+$ be the set of all totally positive algebraic integers over $K$. A $\textit{partition}$ $\lambda$ of a totally positive algebraic integer $\delta$ of $K$ is a finite  sequence of totally positive integers $\lambda_1, \lambda_2, \ldots, \lambda_r$ of $K$ such that $\sum_{i=1}^r\lambda_i=\delta$. Each $\lambda_i$ is called a $part$ of $\lambda$. The partition function $p(\delta)$ is the number of partitions of $\delta$ up to the change of the order. It is known that for all $\delta\in\mathcal O^+$, there are finitely many $\alpha,\beta\in\mathcal O^+$ such that $\alpha+\beta=\delta$ \cite{jkk}. As a consequence, there are finitely many partitions of $\delta$ for each $\delta\in\mathcal O^+$ \cite{jkk}. A partition $\lambda$ is denoted by $(\lambda_1,\lambda_2,\ldots,\lambda_r)$ and we write $\lambda\vdash\delta$ to denote ``$\lambda$ is a partition of $\delta$". For a partition $\lambda=(\lambda_1,\lambda,\ldots,\lambda_r)$ of $\delta$, we write $\lambda=(\delta_1^{n_1}\delta_2^{n_2}\cdots\delta_s^{n_s})$ if $\delta_1,\delta_2,\ldots,\delta_s\in\mathcal O^+$ are pairwise distinct, $\{\lambda_1,\lambda_2,\ldots,\lambda_r\}=\{\delta_1,\delta_2,\ldots,\delta_s\}$ and $n_i$ is the number of $j$ such that $\lambda_j=\delta_i$. 

In \cite{jkk}, formal power $q$-sums are introduced as a generalization of $q$-series. We introduce multiple formal power $q$-sums by a similar method. Assume that $H\subset(\mathcal O^+\cup\{0\})^k$, $\boldsymbol0=(0,0,\ldots,0)\in H$ and $H$ is closed under addition. We write $q^{\alpha}=q_1^{\alpha_1}q_2^{\alpha_2}\cdots q_k^{\alpha_k}$ for $\alpha=(\alpha_1,\alpha_2,\ldots,\alpha_k)\in H$. We define the $\textit{ring of multiple formal q-sums}$, or the $\textit{ring of k variable formal q-sums}$, $\mathbb R[[q]]_H\\=\mathbb R[[q_1,q_2,\ldots,q_k]]_H$ by
\begin{align*}
\left\{\sum_{\delta\in H}c_{\delta}q^{\delta}|c_{\delta}\in\mathbb R\right\}.
\end{align*}
The elements of $\mathbb R[[q]]_H$ are multiple formal power $q$-sums. We also call these $\textit{k variable formal}$ $\textit{power q-sums}$ or $\textit{formal}$ $(q_1,q_2,\ldots,q_k)$$\textit{-sums}$. We do not consider the convergence of (multiple) formal power $q$-sums. Let $f(q)=\sum_{\delta\in H}c_{\delta}q^{\delta}$, $g(q)=\sum_{\delta\in H}d_{\delta}q^{\delta}\in\mathbb R[[q]]_H$. We define the {\it{sum}} and the {\it product} of two multiple formal power $q$-sums $f(q)$ and $g(q)$ by
\begin{align*}
f(q)+g(q)=\sum_{\delta\in H}c_{\delta}q^{\delta}+\sum_{\delta\in H}d_{\delta}q^{\delta}=\sum_{\delta\in H}(c_{\delta}+d_{\delta})q^{\delta}
\end{align*}
and
\begin{align*}
f(q)g(q)=\left(\sum_{\delta\in H}c_{\delta}q^{\delta}\right)\left(\sum_{\delta\in H}d_{\delta}q^{\delta}\right)=\sum_{\gamma\in H}\left(\sum_{\gamma=\alpha+\beta}c_{\alpha}d_{\beta}\right)q^{\gamma}.
\end{align*}
Let $I$ be an index set and $f_i(q)=\sum_{\delta\in H}a_{i,\delta}q^{\delta}\in\mathbb R[[q]]_H$ for all $i\in I$. We define the sum $\sum_{i\in I}f_i(q)$ of multiple formal power $q$-sums $f_i(q)$ only when the condition
\begin{enumerate}
\item[(C1)] $\sum_{\delta\in H}a_{i,\delta}$ converges absolutely for all $i\in I$
\end{enumerate}
is satisfied. Then, we define
\begin{align*}
\sum_{i\in I}f_i(q)=\sum_{i\in I}\left(\sum_{\delta\in H}a_{i,\delta}q^{\delta}\right)=\sum_{\delta\in H}\left(\sum_{i\in I}a_{i,\delta}\right)q^{\delta}.
\end{align*}
We also define the infinite product $\prod_{i\in I}f_i(q)$ of $f_i(q)=\sum_{\delta\in H}a_{i,\delta}q^{\delta}$ under the following two conditions.
\begin{enumerate}
\item[(C2-1)] There are finitely many $i\in I$ such that $a_{i,0}\ne1$. 
\item[(C2-2)] For each $\delta\in H$ there are finitely many $i\in I$ such that $a_{i,\delta}\ne0$. 
\end{enumerate}
In this case, we define $\prod_{i\in I}f_i(q)$ by 
\begin{align*}
\prod_{i\in I}f_i(q)=\prod_{i\in I}\left(\sum_{\delta\in H}a_{i,\delta}q^{\delta}\right)=\sum_{\delta\in H}c_{\delta}q^{\delta}
\end{align*}
where $c_{\delta}$ is the sum $\sum a_{i_1,\delta_1}a_{i_2,\delta_2}\cdots a_{i_k,\delta_k}$ of all products $a_{i_1,\delta_1}a_{i_2,\delta_2}\cdots a_{i_k,\delta_k}$ with $\delta_1,\delta_2,\ldots,$ $\delta_k\in H$ and $i_1,i_2,\ldots, i_k\in I$ satisfying
\begin{enumerate}
\item[1.] $\delta=\delta_1+\delta_2+\cdots+\delta_k$,
\item[2.] $i_1,i_2,\ldots, i_k$ are pairwise distinct, 
\item[3.] $a_{i,0}=1$ for all $i\ne i_1,i_2,\ldots,i_k$. 
\end{enumerate}
Let $f_i(q)=\sum_{\delta\in H}a_{i,\delta}q^{\delta}$ and $g_i(q)=\sum_{\delta\in H}b_{i,\delta}q^{\delta}$ for $i\in I$. Assume that $f_i$ and $g_i$ satisfy the conditions (C2-1) and (C2-2). Then, we have
\begin{align*}
\left(\prod_{i\in I}f_i(q)\right)\left(\prod_{i\in I}g_i(q)\right)=\prod_{i\in I}f_i(q)g_i(q)=\sum_{\delta\in H}c_{\delta}q^{\delta}
\end{align*}
when $c_{\delta}=\sum a_{i_1,\zeta_1}a_{i_2,\zeta_2}\cdots a_{i_k,\zeta_k}b_{j_1,\eta_1}b_{i_j,\eta_2}\cdots b_{j_l,\eta_l}$ is the sum over all $i_1,i_2,$ $\ldots,$ $i_k,j_1,j_2,\ldots, j_l\in I$ and nonzero $\zeta_1,\zeta_2,\ldots,\zeta_k$, $\eta_1,\eta_2,\ldots,\eta_l\in H$ such that 
\begin{enumerate}
\item[1.] $i_1,i_2,\ldots, i_k$ are pairwise distinct,
\item[2.] $j_1,j_2,\ldots, j_l$ are pairwise distinct,
\item[3.] $\delta=\zeta_1+\zeta_2+\cdots+\zeta_k+\eta_1+\eta_2+\cdots+\eta_l$,
\item[4.] $a_{i,0}=1$ for all $i\ne i_1,i_2,\ldots,i_k$,
\item[5.] $b_{j,0}=1$ for all $j\ne j_1,j_2,\ldots,j_l$.
\end{enumerate}
Since $(1-q^{\delta})(1+q^{\delta}+q^{2\delta}+\cdots)=1$, we can write $\frac1{1-q^{\delta}}$ instead of $1+q^{\delta}+q^{2\delta}+\cdots$ in the calculations of multiple formal power $q$-sums identities without considering the convergence. Thus, we have
\begin{align*}
\prod_{\delta\in H}\frac1{1-q^{\delta}}=\prod_{\delta\in H}(1+q^{\delta}+q^{2\delta}+\cdots)\in\mathbb R[[q]]_H.
\end{align*}
Note that a family of (multiple) formal power sums $f_{\delta}(q)=1+q^{\delta}+q^{2\delta}+\cdots$ for $\delta\in H$ satisfies the conditions (C2-1) and (C2-2). 

Let $f_{i,j}(q)=\sum_{\delta\in H} c_{i,j,\delta}q^{\delta}\in\mathbb R[[q]]_H$ for all $i\in I$ and $j\in J$. If $f_{i,j}$ satisfies (C2-1) and (C2-2) for $(i,j)\in I\times J$, then \[\prod_{(i,j)\in I\times J}f_{i,j}(q)\in\mathbb R[[q]]_H\]
and
\[\prod_{(i,j)\in I\times J}f_{i,j}(q)=\prod_{i\in I}\left(\prod_{j\in J}f_{i,j}(q)\right)=\prod_{j\in J}\left(\prod_{i\in I}f_{i,j}(q)\right).\]

Let $q=(q_1,q_2,\ldots,q_k)$, $q'=(q'_1,q'_2,\ldots,q'_l)$, $f=f(q)=\sum_{\alpha\in H}c_{\alpha}q^{\alpha}\in\mathbb R[[q]]_H$ and $g_i=g_i(q')\in\mathbb R[[q']]_{H'}$ for $i=1,2,\ldots,k$ such that $H\subset(\mathcal O^+\cup\{0\})^k$ and $H'\subset(\mathcal O^+\cup\{0\})^l$. Although the composite of two $q$-series is a useful tool in the calculations of $q$-series, it is a serious problem to define the composite $f\circ g=f(g(q'))$ of $f$ and $g=(g_1,g_2,\ldots,g_k)=(g_1(q'),g_2(q'),\ldots,g_k(q'))$ properly. The composite $f\circ g(q')=f(g(q'))=f(g_1,g_2,\ldots,g_k)$ is defined under the assumption
\begin{enumerate}
\item[(C3)] $c_{\alpha}g^{\alpha}=c_{\alpha}g_1^{\alpha_1}g_2^{\alpha_2}\cdots g_k^{\alpha_k}\in\mathbb R[[q']]_{H'}$ for all $\alpha\in H$ and the family \\$c_{\alpha}g^{\alpha}$ for $\alpha\in H$ satisfies (C1).  
\end{enumerate}
The condition (C3) is equvalent to the fact that \[g^{\alpha}=g_1^{\alpha_1}g_2^{\alpha_2}\cdots g_k^{\alpha_k}=\sum_{\beta\in H'}d_{\alpha,\beta}(q')^{\beta}\in\mathbb R[[q']]_{H'}\] for all $\alpha\in H$ such that $c_{\alpha}\ne0$, and for each $\beta\in H'$ $\sum_{\alpha\in H}c_{\alpha}d_{\alpha,\beta}$ converges absolutely. 

\begin{ex}
Let $K=\mathbb Q(\sqrt2)$, $\mathcal O^+=\mathcal O_K^+$, $H=H'=\mathcal O^+\cup\{0\}$, $f(q)=q^{2+\sqrt2}$ and $g(q')=1+q'$. Then $f(g(q'))=(1+q')^{2+\sqrt2}\not\in\mathbb R[[q']]_{H'}$.
\end{ex}

\begin{ex}
Let $K=\mathbb Q$, $\mathcal O^+=\mathbb Z^+$, $f(q)=\frac1{1-q}=1+q+q^2+\cdots$ and $g(q')=1+q'$. Then $f(g(q'))=1+(1+q')+(1+q')^2+\cdots\not\in\mathbb R[[q']]_{\mathbb Z^+\cup\{0\}}$.
\end{ex}

\begin{ex}\label{exhk}
Let $H=(\mathcal O^+\cup\{0\})^k$ and $\gamma_1,\gamma_2,\ldots,\gamma_k\in\mathcal O^+$. If $f(q)\in\mathbb R[[q]]_H$, then $f(q_1^{\gamma_1},q_2^{\gamma_2},\ldots,q_k^{\gamma_k})$ satisfies (C3). Thus $f(q_1^{\gamma_1},q_2^{\gamma_2},\ldots,q_k^{\gamma_k})\in\mathbb R[[q]]_H$.
\end{ex}

\begin{ex}\label{exhl}
Let $H=(\mathcal O^+\cup\{0\})^l$. If 
\begin{align*}
f(q)=f(q_1,q_2,\ldots,q_k)=&\sum_{n=(n_1,n_2,\ldots,n_k)}c_nq^n\\
=&\sum_{n=(n_1,n_2,\ldots,n_k)} c_{n_1,n_2,\ldots,n_k}q_1^{n_1}q_2^{n_2}\cdots q_k^{n_k}
\end{align*}
is a $k$ variable $q$-series and $g_i=g_i(q')=\sum_{\alpha\in H}d_{i,\alpha}(q')^{\alpha}$ such that $d_{i,0}=0$ for all $i=1,2,\ldots,k$, then $f\circ g(q')=f(g_1(q'),g_2(q'),\cdots,g_k(q'))$ satisfies (C3) where $g=(g_1,g_2,\ldots,g_k)$. 
\end{ex}
At present, Examples 3 and 4 are the only known cases such that the composite $f\circ g$ is defined successfully.

Let $\lambda=(\lambda_1,\lambda_2,\ldots,\lambda_r)\vdash\delta$ and $\lambda'=(\lambda_1',\lambda_2',\ldots,\lambda_l')\vdash\delta'$. We define the sum $\lambda\oplus\lambda'=(\lambda_1,\lambda_2,\ldots,\lambda_r,\lambda_1',\lambda_2',\ldots,\lambda_l')$ of $\lambda$ and $\lambda'$. Naturally, $\lambda\oplus\lambda'$ is a partition of  $\delta+\delta'$. Generally, a finite partition sum $\lambda^{(1)}\oplus\lambda^{(2)}\oplus\cdots\oplus\lambda^{(n)}$ can be defined in a similar way. 

\subsection{Primitive Products}
Let $L$ and $K$ be totally real number fields such that $K\subset L$. We call an element $\delta\in\mathcal O^+_L$  \textit{primitive} over $K$, or \textit{L/K-primitive}, if and only if $\delta\not\in\alpha\mathcal O^+_L$ for all non-unit $\alpha\in\mathcal O^+_K$. Let $\mathcal P_{L/K}$ be the set of all $\delta\in\mathcal O^+_L$ primitive over $K$. We call $\mathcal P_{L/K}$ the \textit{primitive~set} of $L$ over $K$, or \textit{L/K-primitive~set}. For each $\delta\in\mathcal O^+_L$, there are $\alpha\in\mathcal O^+_K$ and $\gamma\in\mathcal P_{L/K}$ such that $\delta=\alpha\gamma$. The existence of such $\alpha$ and $\gamma$ is easy to see, but the uniqueness is a serious problem. Let $L=\mathbb Q(\sqrt2,\sqrt3)$, $K=\mathbb Q(\sqrt3)$ and 
\begin{align*}
\delta=&(21+9\sqrt3)+(30+16\sqrt3)(2+\sqrt2)\\
=&(\sqrt3)^2(\sqrt3+1)(2\sqrt3+1)+\sqrt3(\sqrt3+1)^2(2\sqrt3+1)(2+\sqrt2).
\end{align*}
We can see both 
\begin{align*}
\gamma_1=&(6+\sqrt3)+(7+3\sqrt3)(2+\sqrt2)\\
=&\sqrt3(2\sqrt3+1)+(\sqrt3+1)(2\sqrt3+1)(2+\sqrt2)
\end{align*}
and \[\gamma_2=(3+\sqrt3)+(4+2\sqrt3)(2+\sqrt2)=\sqrt3(\sqrt3+1)+(\sqrt3+1)^2(2+\sqrt2)\] are primitive. Therefore $\delta=\alpha_1\gamma_1=\alpha_2\gamma_2$ for \[\alpha_1=3+\sqrt3=\sqrt3(\sqrt3+1)\in\mathcal O^+_K\] and \[\alpha_2=6+\sqrt3=\sqrt3(2\sqrt3+1)\in\mathcal O^+_K.\] Hence the uniqueness of such $\alpha$ and $\gamma$ fails in this case. Assume that $K$ is of class number $1$ and for each $\alpha\in\mathcal O_K$, there is a unit $u$ of $\mathcal O_K$ such that $\alpha u$ is totally positive. We can see that in this case for each $\delta\in\mathcal O^+_L$ there are unique $\alpha\in\mathcal O^+_K$ and $\gamma\in\mathcal P_{L/K}$ such that $\delta=\alpha\gamma$ up to multiplication by a totally positive unit of $\mathcal O_K$. We call $\alpha$ and $\gamma$ the $\textit{scale}$ and the $\textit{primitive~factor}$ of $\delta$, and denote $\alpha=s(\delta)$ and $\gamma=t(\delta)$ respectively. We call the decomposition $\delta=s(\delta)t(\delta)$ the \textit{radial~decomposition} of $\delta$. For each $\gamma\in\mathcal P_{L/K}$, we define the $\gamma$-\textit{section} $\gamma\mathcal O^+_K=\{\gamma\alpha|\alpha\in\mathcal O^+_K\}$.

Let $f(q)=\sum_{\alpha\in\mathcal O^+_K\cup\{0\}}a_{\alpha}q^{\alpha}\in\mathbb R[[q]]_{\mathcal O^+_K\cup\{0\}}$. If $a_0=1$ and for all $\alpha\in\mathcal O^+_K$ there are finitely many units $u$ of $\mathcal O^+_K$ such that $a_{u\alpha}\ne0$, then the product \[\prod_{\gamma\in\mathcal P_{L/K}}f(q^{\gamma})\] can be defined. In other words, $f(q^{\gamma})$ for $\gamma\in\mathcal P_{L/K}$ satisfies the conditions (C2-1) and (C2-2). Let $\tilde{\mathcal P}_{L/K}$ be a set of representatives of all elements of $\{\alpha U^+_K|\alpha\in\mathcal P_{L/K}\}$ where $U^+=U^+_K$ is the set of all totally positive units of $\mathcal O_K$. If $a_0=1$ and $a_{u\gamma}=a_{\gamma}$ for all $\gamma\in\mathcal O^+_K$ and $u\in\ U^+_K$, then the infinite product \[\prod_{\gamma\in\tilde{\mathcal P}_{L/K}}f(q^{\gamma})\] is well-defined. Concretely, the above product is independent of the choice of $\tilde{\mathcal P}_{L/K}$. We call these two products the first and the second $\textit{primitive~product}$ respectively. If $f(q)=\prod_{\alpha\in\mathcal O^+_K}(1-aq^{\alpha})=\sum_{\alpha\in\mathcal O^+_K\cup\{0\}} a_{\alpha}q^{\alpha}$, then we can see that $a_0=1$ and $a_{\alpha}=a_{u\alpha}$ for all $\alpha\in\mathcal O^+_K$ and $u\in U^+_K$. Thus we have \[\prod_{\gamma\in\tilde{\mathcal P}_{L/K}}\left(\prod_{\alpha\in\mathcal O^+_K}(1-aq^{\alpha\gamma})\right)=\prod_{\delta\in\mathcal O^+_L}(1-aq^{\delta}).\] If $K=\mathbb Q$, then since $U^+_{\mathbb Q}=\{1\}$, $\tilde{\mathcal P}_{L/\mathbb Q}=\mathcal P_{L/\mathbb Q}$ and $\prod_{\gamma\in\mathcal P_{L/\mathbb Q}}f(q^{\gamma})$ is well-defined for $f(q)=\sum_{\alpha\in\mathcal O^+\cup\{0\}}a_{\alpha}q^{\alpha}$ with $a_0=1$.

\subsection{Primitive Products over $\mathbb Q$}
In this subsection, we treat a special case of the previous subsection that the base field is $\mathbb Q$. Let $K$ be a totally real number field. Let $\mathcal P=\mathcal P_{K/\mathbb Q}$ be the $\mathcal P_{K/\mathbb Q}$-primitive set.  As noted at the end of the last subsection, the product \[\prod_{\gamma\in\mathcal P_{K/\mathbb Q}}f(q^{\gamma})\] is well-defined. Let $\mathcal P_i=\cup_{j=1}^{i}(j\mathcal P)$ and $\mathcal P_{\infty}=\cup_{j=1}^{\infty}(j\mathcal P)$ where $j\mathcal P=\{j\gamma|\gamma\in\mathcal P\}$. Since $\mathcal O^+_{\mathbb Q}=\mathbb Z^+$, the $\gamma$-section is $\gamma\mathbb Z^+=\{n\gamma|n\in\mathbb Z^+\}$ and $\cup_{\gamma\in\mathcal P}\gamma\mathbb Z^+=\mathcal O^+_K$. We can easily see that $\mathcal P=\{1\}$ when $K=\mathbb Q$, and otherwise $\mathcal P$ is an infinite set. 

\begin{lem}\label{ug}
If $\gamma,\gamma'\in\mathcal P$ and $\gamma\ne\gamma'$, then $\gamma\mathbb Z^+\cap\gamma'\mathbb Z^+=\emptyset$.
\end{lem}
\begin{proof}
If $\gamma\mathbb Z^+\cap\gamma'\mathbb Z^+\not=\emptyset$, then $n\gamma=n'\gamma'\in\gamma\mathbb Z^+\cap\gamma'\mathbb Z^+$ for some $n,n'\in\mathbb Z^+$. We may assume that $n$ and $n'$ are relatively prime. Then  $nx+n'y=1$ for some $x,y\in\mathbb Z$.  Let $\beta={\gamma\over n'}={\gamma'\over n}\in K$. Since $\beta=(nx+n'y)\beta=x\gamma'+y\gamma\in\mathcal O^+$ and both $\gamma=n'\beta $ and $\gamma'=n\beta$ are primitive, we have $n=n'=1$, and thus $\gamma=\gamma'$. We have the desired contradiction.
\end{proof}

For a primitive element $\gamma$ of $\mathcal O^+$, we call a partition $\lambda=(\lambda_1,\lambda_2,\ldots,\lambda_r)$ a $\gamma$-\textit{sectional~partition} if $\lambda_i\in\gamma\mathbb Z^+$ for all $i=1,2,\ldots,r$. For all partition $\lambda=(\lambda_1,\lambda_2,\ldots,\lambda_r)$ of $\delta$, there are pairwise distinct $\gamma_1,\gamma_2,\ldots,\gamma_l\in\mathcal P$ and $\gamma_i$-sectional partitions $\lambda^{(i)}$ such that $\lambda=\lambda^{(1)}\oplus\lambda^{(2)}\oplus\cdots\oplus\lambda^{(l)}$, which we call the \textit{sectional~decomposition} of $\lambda$. If all parts of a partition $\lambda$ are of scale $1$, or equivalently all parts of $\lambda$ are primitive, then we call $\lambda$ a \text{primitive~partition} of $\delta$. Note that if $\lambda$ is a primitive partition, then we have $\lambda=(\gamma_1^{n_1}\gamma_2^{n_2}\ldots\gamma_l^{n_l})$ for some pairwise distinct primitive elements $\gamma_1,\gamma_2,\ldots,\gamma_l$ of $\mathcal O^+_K$ and $n_1,n_2,\ldots,n_l\in\mathbb Z^+$. 

Let $\lambda=(\lambda_1,\lambda_2,\ldots,\lambda_r)$ and $\lambda'$ be partitions of $\delta$. We call $\lambda'$ a \textit{refinement} of $\lambda$ if $\lambda'=\lambda^{(1)}\oplus\lambda^{(2)}\oplus\cdots\oplus\lambda^{(r)}$ for some partitions $\lambda^{(i)}$ of $\lambda_i$. Moreover, if $\lambda^{(i)}$ is a sectional parition for all $i$, then we call $\lambda'$ a \textit{sectional~refinement} of $\lambda$. In this case, $\lambda_i=n_i\gamma_i$ for some $n_i\in\mathbb Z^+$ and $\gamma_i\in\mathcal P$, and $\lambda^{(i)}=(k_{i,1}\gamma_i,k_{i,2}\gamma_i,\ldots,k_{i,l_i}\gamma_i)$ for some $k_{i,1},k_{i,2},\ldots,k_{i,l_i}$ such that $k_{i,1}+k_{i,2}+\cdots+k_{i,l_i}=n_i$ for all $i$.
 
A sequence $\alpha_1,\alpha_2,\ldots,\alpha_k$ of totally positive integers are $\textit{consecutive}$ if the primitive factors $t(\alpha_i)$ of $\alpha_i$ for $i=1,2,\ldots,k$ are equal and the scales $s(\alpha_i)$ of $\alpha_i$ for $i=1,2,\ldots,k$ are actually consecutive positive rational integers. More concretely, $\alpha_1,\alpha_2,\ldots,\alpha_k$ are consecutive if $t(\alpha_1)=t(\alpha_2)=\cdots=t(\alpha_k)$ and $s(\alpha_2)-s(\alpha_1)=s(\alpha_3)-s(\alpha_2)=\cdots=s(\alpha_k)-s(\alpha_{k-1})=1$. In this case we have $\alpha_i=(m+i-1)\gamma$ for all $i$ where $\gamma=t(\alpha_1)$ and $m=s(\alpha_1)$. A partition $\lambda=(\lambda_1,\lambda_2,\ldots,\lambda_k)$ is consecutive if $\lambda_1,\lambda_2,\ldots,\lambda_k$ is consecutive by some suitable change of the order. A partition $\lambda$ is $k$-$\textit{noncontiguous}$ if $\lambda$ is a sum of $k$ consecutive partitions and not a sum of $(k-1)$ consecutive partitions, or equivalently  $\lambda=\lambda^{(1)}\oplus\lambda^{(2)}\oplus\cdots\oplus\lambda^{(k)}$ such that $\lambda^{(i)}$ is consecutive for all $i=1,2\ldots,k$ and $\lambda^{(i)}\oplus\lambda^{(j)}$ is not consecutive for all distinct $i$ and $j$.

Let $\mathscr{S}$ denote the set of all partitions of all $\delta\in\mathcal O^+$. Let $p(S,\delta)$ denote the number of partitions of $\delta$ that belong to a subset $S$ of the set $\mathscr S$. Let $H$ be a set of totally positive integers over $K$. Let $''H''$ and $''H''(\le d)$ denote the set of all partitions whose parts lie in $H$ and that in which no part appears more than $d$ times and each part is in $H$ respectively. Let $K'$ be a subfield of $K$ and $p_{K'}(\delta)$ denote the number of partitions of $\delta\in\mathcal O^+_{K'}$ where each part is an element of $K'$. In particular, the number of partitions of $\delta$ such that all parts are natural numbers is $p('{'\mathbb Z^+}'', \delta)$ or $p_{\mathbb Q}(\delta)$. For a guidance of the partition theory of positive rational integers, see \cite{and11}. This paper uses the analogous notations introduced in this book. 

Let $\gamma$ be a primitive element of $\mathcal O^+_K$. The following lemma follows from the natural one-to-one correspondence from the set of all partitions of $n\in\mathbb Z^+$ onto the set of all $\gamma-$sectional partitions of $n\gamma$ which maps \[\lambda=(\lambda_1,\lambda_2,\ldots,\lambda_k)\vdash n\] to \[(\lambda_1\gamma,\lambda_2\gamma,\ldots,\lambda_k\gamma)\vdash n\gamma.\]

\begin{lem}\label{pgnn}
The number of $\gamma$-sectional paritions of $n\gamma$ is equal to that of partitions of $n$ in $\mathbb Z^+$. In other words, $p({''{\gamma\mathbb Z^+}''},n\gamma)=p('{'\mathbb Z^+}'', n)=p_{\mathbb Q}(n)$ for all $n\in\mathbb Z^+$ and $\gamma\in\mathcal P$.
\end{lem}

Let $P_{\delta}$ be the set of all primitive partitions of $\delta$ and $\lambda\in P_{\delta}$. Let $S_{\lambda}$ be the set of all partitions $\lambda'$ such that $\lambda$ is a sectional refinement of $\lambda'$. We can easily see that a primitive partition $(\gamma_1^{n_1}\gamma_2^{n_2}\cdots\gamma_l^{n_l})$ of $\delta$ is a sectional refinement of $\lambda^{(1)}\oplus\lambda^{(2)}\oplus\cdots\oplus\lambda^{(l)}\vdash\delta$ where $\lambda^{(i)}\in{''\gamma_i{\mathbb Z^+}''}$ and $\lambda^{(i)}\vdash n_i\gamma_i$ for all $i$. It is not difficult to see that for each partition $\lambda$ of $\delta\in\mathcal O^+$ there exists a unique primitive sectional refinement of $\lambda$. As a consequence the set of all partitions of $\delta\in\mathcal O^+$ is the disjoint union $\cup S_{\lambda}$ of all $S_{\lambda}$ for all primitive partitions $\lambda$ of $\delta$.

\begin{lem}\label{Slambda}
For each $\lambda=(\gamma_1^{n_1}\gamma_2^{n_2}\ldots\gamma_l^{n_l})\in P_{\delta}$, we have 
\[|S_{\lambda}|=\prod_{h=1}^lp(''{\gamma_h\mathbb Z^+}'',n_h\gamma_h)=\prod_{h=1}^lp(''{\mathbb Z^+}'',n_h)=\prod_{h=1}^lp_{\mathbb Q}(n_h).\] Therefore, for all $\delta\in\mathcal O^+$, 
\begin{align*}
p(\delta)=\sum_{\lambda\in P_{\delta}}|S_{\lambda}|=\sum_{\lambda\in P_{\delta}}\left(\prod_{h=1}^lp_{\mathbb Q}(n_h)\right)
\end{align*}
where $\lambda=(\gamma_1^{n_1}\gamma_2^{n_2}\ldots\gamma_l^{n_l})$.
\end{lem}
\begin{proof}
Let $\lambda^{(h)}=(\gamma_h^{n_h}$) for all $h$. Then $\lambda^{(h)}$ is a primitive partition of $n_h\gamma_h$ for all $h$, and $\lambda=\lambda^{(1)}\oplus\lambda^{(2)}\oplus\cdots\oplus\lambda^{(l)}$. A partition of $n_h\gamma_h$ is a refinement of $\lambda^{(h)}$ if and only if it is a $\gamma_h$-sectional partition for all $h$. Thus we have $|S_{\lambda^{(h)}}|=p(''{\gamma_h\mathbb Z^+}'',n_h\gamma_h)$. By a natural one-to-one correspondence between $S_{\lambda}$ onto $S_{\lambda^{(1)}}\times S_{\lambda^{(2)}}\times\cdots\times S_{\lambda^{(l)}}$ which maps $\sigma^{(1)}\oplus\sigma^{(2)}\oplus\cdots\oplus\sigma^{(l)}$ to $(\sigma^{(1)},\sigma^{(2)},\ldots,\sigma^{(l)})$ and by Lemma \ref{pgnn}, we have 
\begin{align*}
|S_{\lambda}|&=|S_{\lambda^{(1)}}\times S_{\lambda^{(2)}}\times\cdots\times S_{\lambda^{(l)}}|=\prod_{h=1}^l|S_{\lambda^{(h)}}|\\
&=\prod_{h=1}^lp(''{\gamma_h\mathbb Z^+}'',n_h\gamma_h)=\prod_{h=1}^lp_{\mathbb Q}(n_h).
\end{align*}
If $\lambda,\lambda'\in P_{\delta}$ and $\lambda\ne\lambda'$, then by Lemma \ref{ug}, we have $S_{\lambda}\cap S_{\lambda'}=\emptyset$. Since $p(\delta)=\sum_{\lambda\in P_{\delta}}|S_{\lambda}|$, we have \[p(\delta)=\sum_{\lambda\in P_{\delta}}\left(\prod_{h=1}^lp(''{\gamma_h\mathbb Z^+}'',n_h\gamma_h)\right)=\sum_{\lambda\in P_{\delta}}\left(\prod_{h=1}^lp_{\mathbb Q}(n_h)\right)\] where $\lambda=(\gamma_1^{n_1}\gamma_2^{n_2}\ldots\gamma_l^{n_l})$.
\end{proof}

Let $\mathscr A$ be a subset of $\mathscr S$. We call $\mathscr A$ closed under $\oplus$ if $\lambda\oplus\lambda'\in\mathscr A$ for all $\lambda,\lambda'\in\mathscr A$. We call two partitions $\lambda$ and $\lambda'$ have $\textit{disjoint~sections}$ if every part $\lambda_i$ of $\lambda$ and $\lambda_j'$ of $\lambda'$ have different sections, i.e., $t(\lambda_i)\ne t(\lambda_j')$ for all $i$ and $j$. We call $\mathscr A$ closed under $\oplus$ for disjoint sectional partitions if $\lambda\oplus\lambda'\in\mathscr A$ for all $\lambda,\lambda'\in\mathscr A$ which have disjoint sections. The following corollary is a simple extension of Lemma \ref{Slambda}.

\begin{cor}\label{psp}
Let $\mathscr C\subset\mathscr S$ be closed under $\oplus$ for disjoint sectional partitions. Then, for all $\delta\in\mathcal O^+$, we have
\begin{align*}
p(\mathscr C,\delta)=\sum_{\lambda\in P_{\delta}}\left(\prod_{h=1}^lp(\mathscr C\cap{''{\gamma_h\mathbb Z^+}''},n_h\gamma_h)\right),
\end{align*}
where $\lambda=(\gamma_1^{n_1}\gamma_2^{n_2}\ldots\gamma_l^{n_l})$. Therefore, for $H\subset\mathcal O^+$ and for all $\delta\in\mathcal O^+$, we have 
\[p(''H'',\delta)=\sum_{\lambda\in P_{\delta}}\left(\prod_{h=1}^lp(''{H\cap\gamma_h\mathbb Z^+}'',n_h\gamma_h)\right)\]
where $\lambda=(\gamma_1^{n_1}\gamma_2^{n_2}\ldots\gamma_l^{n_l})$.
\end{cor}

\subsection{Basic Hypergeometric Series}
Let $S$ be a subset of $\mathcal O^+\cup\{0\}$ and $a\in\mathbb R$. We define $(a;q)_S=\prod_{s\in S}(1-aq^s)$. We can see that $1-aq^s$ for $s\in S$ satisfies (C2-1) and (C2-2), and thus we have $(a;q)_S\in\mathbb R[[q]]_{\mathcal O^+\cup\{0\}}$. This is a generalization of $(a;q)_n=\prod_{i=0}^{n-1}(1-aq^i)$ and $(a;q)_{\infty}=\prod_{i=0}^{\infty}(1-aq^i)$, which play important roles in various $q$-series identities. We define $(a;q)_n^*=\frac{(a;q)_n}{1-a}=\prod_{i=1}^{n-1}(1-aq^i)$ and $(a;q)_{\infty}^*=\frac{(a;q)_{\infty}}{1-a}=\prod_{i=1}^{\infty}(1-aq^i)$ for $a\ne1$. In fact $(a;q)_n=(a;q)_{\{0,1,\ldots,n-1\}}$, $(a;q)_n^*=(a;q)_{\{1,2,\ldots,n-1\}}$, $(a;q)_{\infty}=(a;q)_{\mathbb Z^+\cup\{0\}}$ and $(a;q)_{\infty}^*=(a;q)_{\mathbb Z^+}$. A consequence of Example \ref{exhk} is that $(a;q^{\gamma})_n,(a;q^{\gamma})_{\infty},(a;q^{\gamma})_n^*,(a;q^{\gamma})_{\infty}^*\in\mathbb R[[q]]_{\mathcal O^+\cup\{0\}}$ for all $\gamma\in\mathcal O^+$. Since $\mathcal P_m=\cup_{n=1}^m(n\mathcal P)$ and $\mathcal O^+=\cup_{n=1}^{\infty}(n\mathcal P)$, we have
\begin{align*}
(a;q)_{\mathcal P_m}&=\prod_{\delta\in\mathcal P_m}(1-aq^{\delta})\\
&=\left(\prod_{\delta\in\mathcal P}(1-aq^\delta)\right)\left(\prod_{\delta\in2\mathcal P}(1-aq^\delta)\right)\cdots\left(\prod_{\delta\in m\mathcal P}(1-aq^\delta)\right)\\
&=\prod_{n=1}^{m}\left(\prod_{\gamma\in\mathcal P}(1-aq^{n\gamma})\right)=\prod_{\gamma\in\mathcal P}\left(\prod_{n=1}^{m}(1-aq^{n\gamma})\right)=\prod_{\gamma\in\mathcal P}(a;q^{\gamma})_{m+1}^*,\\
(a;q)_{\mathcal O^+}&=\prod_{\delta\in\mathcal O^+}(1-aq^{\delta})=\prod_{n=1}^{\infty}\left(\prod_{\delta\in n\mathcal P}(1-aq^{\delta})\right)=\prod_{n=1}^{\infty}\left(\prod_{\gamma\in\mathcal P}(1-aq^{n\gamma})\right)\\
&=\prod_{\gamma\in\mathcal P}\left(\prod_{n=1}^{\infty}(1-aq^{n\gamma})\right)=\prod_{\gamma\in\mathcal P}(a;q^{\gamma})_{\infty}^*.
\end{align*}
Similarly, we have \[(a;q)_{\mathcal P_m}^k=\prod_{\gamma\in\mathcal P}((a;q^{\gamma})_{m+1}^*)^k\] and \[(a;q)_{\mathcal O^+}^k=\prod_{\gamma\in\mathcal P}((a;q^{\gamma})_{\infty}^*)^k\] for all $k\in\mathbb Z$. Note that all of three families of $q$-sums $1-aq^{\delta}$ for $\delta\in\mathcal O^+$, $1-aq^{n\gamma}$ for $n\in\mathbb Z^+$, $\gamma\in\mathcal P$ and $\prod_{n=1}^{\infty}(1-aq^{n\gamma})=(a;q^{\gamma})_{\infty}^*$ for $\gamma\in\mathcal P$ satisfy (C2-1) and (C2-2) although a family $\prod_{n=0}^{\infty}(1-aq^{n\gamma})=(a;q^{\gamma})_{\infty}$ of $q$-sums for $\gamma\in\mathcal P$ does not satisfy (C2-1) when $a\ne1$.

We introduce a method to extend some $q$-series identities over $\mathbb Q$ to formal power $q$-sum identities over $K$. 

\begin{lem}\label{fg1q}
Let $m_1(n), m_2(n),\ldots, m_k(n)$ be  functions for $n$ and $\sigma_i,\sigma_j'$ be $1$ or $-1$ for all $i=1,2,\ldots,k$ and $j=1,2,\ldots,l$ respectively. Suppose $a_i=a_i(q), b_j=b_j(q)$ and $C(n)=C(n)(q)$ belong to $\mathbb R[[q]]_H$ for $q=(q_1,q_2,\ldots,q_s)=(q_1,q')$ and $H=(\mathcal O^+\cup\{0\})^s$ such that constant terms of $a_i$, $b_j$ and $C(n)$ are all $0$ for $i,j,n\ge1$. If 
\begin{align}\label{fgid}
\sum_{n=0}^{\infty}\left(\prod_{i=1}^k(a_i;q_1)_{m_i(n)}^{\sigma_i}\right)C(n)=\prod_{j=1}^l(b_j;q_1)_{\infty}^{\sigma_j'},
\end{align}
then for all $\mathcal P'\subset\mathcal P$, we have 
\begin{align*}
&\prod_{\gamma\in\mathcal P'}\left(\sum_{n=0}^{\infty}\left(\prod_{i=1}^k(a_i(q_1
^{\gamma},q');q_1^{\gamma})_{m_i(n)}^{\sigma_i}\right)C(n)(q_1^{\gamma},q')\right)\\
=&\prod_{\gamma\in\mathcal P'}\left(\prod_{i=1}^l(b_i(q_1^{\gamma},q');q_1^{\gamma})_{\infty}^{\sigma_i'}\right).
\end{align*}
\end{lem}
\begin{proof}
Let $f(q)=\sum_{n=0}^{\infty}\left(\prod_{i=1}^k(a_i;q_1)_{m_i(n)}^{\sigma_i}\right)C(n)$ and $g(q)=\prod_{j=1}^l(b_j;q_1)_{\infty}^{\sigma_j'}$. 
Suppose that \[f(q)=g(q)=\sum_{m=0}^{\infty}c_mq_1^m\] where $c_m\in\mathbb R[[q']]_{H'}$ for $H'=(\mathcal O^+\cup\{0\})^{s-1}$. Then for each $\gamma\in\mathcal P'\subset\mathcal P$, we have \[f(q_1^{\gamma},q')=g(q_1^{\gamma},q')=\sum_{m=0}^{\infty}c_mq_1^{\gamma m}.\]  Since \[\prod_{i=1}^n(b_i;q_1)_{\infty}^{\sigma_i'}=\prod_{i=1}^n\left(\prod_{j=1}^{\infty}(1-b_iq_1^j)\right)^{\sigma_i'}=\sum_{m=0}^{\infty}c_mq_1^m,\] we have $c_0=1$. For each $\gamma\in\mathcal P'$, \[f(q_1^{\gamma},q')=g(q_1^{\gamma},q')=\sum_{m=0}^{\infty}c_mq_1^{\gamma m}=\sum_{\delta\in\mathcal O^+\cup\{0\}}a_{\gamma,\delta}q_1^{\delta}\] where $a_{\gamma,\delta}=c_m$ if $\delta=\gamma m$ and otherwise $a_{\gamma,\delta}=0$. Thus $a_{\gamma,0}=1$ for all $\gamma\in\mathcal P'$, and for each $\delta\in\mathcal O^+\cup\{0\}$ there is at most one $\gamma\in\mathcal P'$ such that $a_{\gamma,\delta}\ne0$. As a consequence, a family $\sum_{\delta\in\mathcal O^+\cup\{0\}}a_{\gamma,\delta}q_1^{\delta}$ for $\gamma\in\mathcal P'$ of $q$-sums satisfies (C2-1) and (C2-2). Thus we have 
\begin{align*}
&\prod_{\gamma\in\mathcal P'}\left(\sum_{n=0}^{\infty}\left(\prod_{i=1}^k(a_i(q_1^{\gamma},q');q_1^{\gamma})_{m_i(n)}^{\sigma_i}\right)C(n)(q_1^{\gamma},q')\right)\\
=&\prod_{\gamma\in\mathcal P'}f(q_1^{\gamma},q')=\prod_{\gamma\in\mathcal P'}\left(\sum_{\delta\in\mathcal O^+\cup\{0\}}a_{\gamma,\delta}q_1^{\delta}\right)\\
=&\prod_{\gamma\in\mathcal P'}g(q_1^{\gamma},q')=\prod_{\gamma\in\mathcal P'}\left(\prod_{i=1}^l(b_i(q_1^{\gamma},q');q_1^{\gamma})_{\infty}^{\sigma_i'}\right).
\end{align*}
\end{proof}

The following lemma is a consequence of Examples \ref{exhk} and \ref{exhl}.
\begin{lem}
If $a=a(q)\in\mathbb R[[q]]_H$ such that the costant term of $a(q)$ is $0$, then for each $\gamma\in\mathcal P$, $(a(q_1^{\gamma},q');q_1^{\gamma})_n,(a(q_1^{\gamma},q');q_1^{\gamma})_{\infty}\in\mathbb R[[q]]_H$ where $q=(q_1,q_2,\ldots,q_s)=(q_1,q')$.
\end{lem}

There are many important $q$-series identities of the form (\ref{fgid}) given in Lemma \ref{fg1q}. Among these, we generalize Cauchy's Theorem [\cite{and11}, Theorem 2.1]: 
\begin{align}\label{cau}
\frac{(at;q)_{\infty}}{(t;q)_{\infty}}=\sum_{n=0}^{\infty}\frac{(a;q)_{n-1}t^n}{(q)_n}
\end{align}
and LeVeque's Theorem [\cite{and11}, Corollary 2.7]:
\begin{align}\label{leve}
(aq;q^2)_{\infty}(q)_{\infty}=\sum_{n=0}^{\infty}\frac{(a;q)_nq^{\frac{n(n+1)}2}}{(q)_n}.
\end{align}

The following two corollaries are obtained from the identities (\ref{cau}) and (\ref{leve}) by applying Lemma \ref{fg1q}.

\begin{cor}
Let $H=(\mathcal O^+\cup\{0\})^n$ and $q=(q_1,a,t)$. Then we have
\begin{align*}
\prod_{\gamma\in\mathcal P}\frac{(a;q_1^{\gamma})_{\infty}}{(t;q_1^{\gamma})_{\infty}}=\prod_{\gamma\in\mathcal P}\left(\sum_{n=0}^{\infty}\frac{(a;q_1^{\gamma})_{n-1}t^n}{(q_1^{\gamma})_n}\right).
\end{align*}
\end{cor}

This identity follows from Lemma \ref{fg1q} for $\mathcal P'=\mathcal P$, $s=3, q=(q_1,a,t), k=2, a_1=a, a_2=q_1, \sigma_1=1, \sigma_2=-1, l=2, b_1=at, b_2=t, \sigma_1'=1, \sigma_2'=-1$ and $C(n)=t^n$. Since $\prod_{\delta\in\mathcal O^+\setminus2\mathcal O^+}(1-aq^{\delta})=\prod_{\gamma\in\mathcal P}(aq^{\gamma};q^{2\gamma})_{\infty}$, by Lemma \ref{fg1q} we have the following corollary.

\begin{cor}\label{lev}
We have
\begin{align*}
\prod_{\delta\in\mathcal O^+}\frac{(1-aq^{\delta})(1-q^{\delta})}{(1-aq^{2\delta})}=&\prod_{\gamma\in\mathcal P}(aq^{\gamma};q^{2\gamma})_{\infty}(q^{\gamma};q^{\gamma})_{\infty}\\
=&\prod_{\gamma\in\mathcal P}\left(\sum_{n=0}^{\infty}\frac{(a;q^{\gamma})_nq^{\frac{n(n+1)\gamma}2}}{(q^{\gamma};q^{\gamma})_n}\right).
\end{align*}
\end{cor}

\section{Main Theorems}
\subsection{The Sylvester Theorem over Totally Real Number Fields}
The following theorem was obtained by Sylvester \cite{syl} by improving Euler's theorem \cite{euler}.
\begin{thm}[\cite{and11}, Theorem 2.12, Sylvester]\label{syl}
Let $A_k(n)$ denote the number of partitions of $n$ into odd parts such that exactly $k$ different parts occur. Let $B_k(n)$ denote the number of partitions $\lambda=(\lambda_1,\lambda_2,\ldots,\lambda_r)$ of $n$ such that the sequence $(\lambda_1,\lambda_2,\ldots,\lambda_r)$ is composed of exactly $k$ noncontiguous sequences of one or more consecutive integers. Then $A_k(n)=B_k(n)$ for all $k$ and $n$.
\end{thm}
By the proof of above theorem, we have 
\begin{align}\label{sy}
\sum_{k=0}^{\infty}\sum_{n=0}^{\infty}A_k(n)a^kq^n=((1-a)q;q^2)_{\infty}(-q)_{\infty}=\sum_{k=0}^{\infty}\sum_{n=0}^{\infty}B_k(n)a^kq^n.
\end{align}
To generalize Theorem \ref{syl}, we define two sets of partitions of $\delta\in\mathcal O^+$ corresponding to $A_k(n)$ and $B_k(n)$ respectively. Let $\mathscr A_k(\delta)$ be the set of all partitions of $\delta$ into odd scaled parts such that exactly $k$ different parts occur. Let $\mathscr B_k(\delta)$ be the set of all $k$-noncontiguous partitions of $\delta$. Let $\mathscr A_k=\cup_{\delta\in\mathcal O^+}\mathscr A_k(\delta)$ and $\mathscr B_k=\cup_{\delta\in\mathcal O^+}\mathscr B_k(\delta)$.

\begin{thm}\label{pA=pB}
We have $p(\mathscr A_k,\delta)=p(\mathscr B_k,\delta)$ for all $k$ and $\delta\in\mathcal O^+$.
\end{thm}
\begin{proof}
For each $\gamma\in\mathcal P$, let $\mathscr A_{k,\gamma}$ and $\mathscr B_{k,\gamma}$ be the subsets of $\mathscr A_k$ and $\mathscr B_k$ respectively whose elements are all $\gamma$-sectional partitions. In fact $\mathscr A_{k,\gamma}=\mathscr A_k\cap{''{\gamma\mathbb Z^+}''}$ and $\mathscr B_{k,\gamma}=\mathscr B_k\cap{''{\gamma\mathbb Z^+}''}$. Let $\mathscr A_{k,\gamma}(\delta)=\mathscr A_k(\delta)\cap{''{\gamma\mathbb Z^+}''}$ and $\mathscr B_{k,\gamma}(\delta)=\mathscr B_k(\delta)\cap{''{\gamma\mathbb Z^+}''}$. If $\lambda\in\mathscr A_k(\delta)$, then there is a primitive partition $\mu=(\gamma_1^{n_1}\gamma_2^{n_2}\cdots\gamma_l^{n_l})$ of $\delta$ which is a refinement of $\lambda$. Thus, $\lambda=\lambda^{(1)}\oplus\lambda^{(2)}\oplus\cdots\oplus\lambda^{(l)}$ for some partitions $\lambda^{(1)},\lambda^{(2)},\ldots,\lambda^{(l)}$ such that $\lambda^{(i)}$ is a $\gamma_i$-sectional partition of $n_i\gamma_i$. Also each part of $\lambda^{(i)}$ is of scale odd. If $k_i$ is the number of different parts of $\lambda^{(i)}$, then $\sum_{i=1}^lk_i=k$. Hence $\mathscr A_k(\delta)$ is the union of \[S_{\mu,k_1,k_2,\ldots,k_l}=\{\lambda=\lambda^{(1)}\oplus\lambda^{(2)}\oplus\cdots\oplus\lambda^{(l)}\vdash\delta|\lambda^{(h)}\in\mathscr A_{k_h,\gamma_h}~{\rm for~all~} h\}\] for all primitive partition $\mu=(\gamma_1^{n_1}\gamma_2^{n_2}\cdots\gamma_l^{n_l})$ of $\delta$ and $k_1,k_2,\ldots,k_l\in\mathbb Z^+$ such that $k_1+k_2+\cdots+k_l=k$. Suppose that a primitive partition $\mu=(\gamma_1^{n_1}\gamma_2^{n_2}\cdots\gamma_l^{n_l})$ of $\delta$ and $k_1,k_2,\ldots,k_l\in\mathbb Z^+$ such that $k_1+k_2+\cdots+k_l=k$ are given. Since $p(\mathscr A_{k_h,\gamma_h},n_h\gamma_h)$ is the number of partitions $\lambda^{(h)}\in\mathscr A_{k_h,\gamma_h}(n_h\gamma_h)$, the number $|S_{\mu,k_1,k_2,\ldots,k_l}|$ of $\lambda=\lambda^{(1)}\oplus\lambda^{(2)}\cdots\oplus\lambda^{(l)}$ such that $\lambda^{(h)}\in\mathscr A_{k_h,\gamma_h}(n_h\gamma_h)$ is $\prod_{h=1}^lp(\mathscr A_{k_h,\gamma_h},n_h\gamma_h)$. Hence the number of $\lambda\in\mathscr A_k(\delta)$ such that $\mu$ is a refinement of $\lambda$ is $\sum_{k_1+k_2+\cdots+k_l=k}(\prod_{h=1}^lp(\mathscr A_{k_h,\gamma_h},n_h\gamma_h))$. Suppose that $\lambda\in\mathscr B_k(\delta)$ and $\mu=(\gamma_1^{n_1}\gamma_2^{n_2}\cdots\gamma_l^{n_l})$ is a primitive refinement of $\lambda$. Then, $\lambda=\sigma^{(1)}\oplus\sigma^{(2)}\oplus\cdots\oplus\sigma^{(k)}$ for some consecutive partitions $\sigma^{(1)},\sigma^{(2)},\ldots,\sigma^{(k)}$. Of course $\sigma^{(i)}\oplus\sigma^{(j)}$ is not consecutive for all distinct $i$ and $j$. For each $i$, since $\sigma^{(i)}$ is consecutive, we have $\sigma^{(i)}=(s_i\gamma_i',(s_i+1)\gamma_i',\ldots,(s_i+t_i)\gamma_i')$ for some $s_i\ge1$, $t_i\ge0$ and $\gamma_i'\in\mathcal P$. Thus $\sigma^{(i)}$ is a $\gamma_i'$-sectional partition. Since $\mu=(\gamma_1^{n_1}\gamma_2^{n_2}\cdots\gamma_l^{n_l})$ is a refinement of $\lambda$, $\gamma_1,\gamma_2,\ldots,\gamma_l$ are distinct elements of  $\mathcal P$ among $\gamma_1',\gamma_2',\ldots,\gamma_k'$. Let $T_h$ be the set of all $i$ such that $\gamma_i'=\gamma_h$ for each $h$. Let $k_h$ be the number of elemnets of $T_h$. Obviously $\sum_{h=1}^lk_h=k$. Since $\mu=(\gamma_1^{n_1}\gamma_2^{n_2}\cdots\gamma_l^{n_l})$ is a refinement of $\lambda$ and all the parts which belong to $\gamma_h$-section $''{\gamma_h\mathbb Z^+}''$ are the parts of $\sigma^{(i)}$ for some $i\in T_h$, we have $\sum_{i\in T_h}\left(\sum\sigma^{(i)}\right)=\sum_{i\in T_h}\frac{(t_i+1)(2s_i+t_i)}2\gamma_h=n_h\gamma_h$ where $\sum\sigma^{(i)}$ is the sum of all parts of $\sigma^{(i)}$. For each $h$, let $\sigma^{(h,1)},\sigma^{(h,2)},\ldots,\sigma^{(h,k_h)}$ be a rearrangement of all $\sigma^{(i)}$ for $i\in T_h$. Let $\tilde\sigma^{(h)}=\sigma^{(h,1)}\oplus\sigma^{(h,2)}\oplus\cdots\oplus\sigma^{(h,k_h)}$. Then $\lambda=\tilde\sigma^{(1)}\oplus\tilde\sigma^{(2)}\oplus\cdots\oplus\tilde\sigma^{(l)}$ and $\tilde\sigma^{(h)}$ is a $k_h$-noncontiguous $\gamma_h$-sectional partition of $n_h\gamma_h$. In other words, $\lambda=\tilde\sigma^{(1)}\oplus\tilde\sigma^{(2)}\oplus\cdots\oplus\tilde\sigma^{(l)}$ for some $\tilde\sigma^{(1)},\tilde\sigma^{(2)},\cdots,\tilde\sigma^{(l)}$ such that $\tilde\sigma^{(h)}\in\mathscr B_{k_h,\gamma_h}(n_h\gamma_h)$. Therefore, $\mathscr B_k(\delta)$ is the union of \[S'_{\mu,k_1,k_2,\ldots,k_l}=\{\lambda=\tilde\sigma^{(1)}\oplus\tilde\sigma^{(2)}\oplus\cdots\oplus\tilde\sigma^{(l)}\vdash\delta|\tilde\sigma^{(h)}\in\mathscr B_{k_h,\gamma_h}(n_h\gamma_h)~{\rm for~all~} h\}\] for all primitive partition $\mu=(\gamma_1^{n_1}\gamma_2^{n_2}\cdots\gamma_l^{n_l})$ of $\delta$ and $k_1,k_2,\ldots,k_l\in\mathbb Z^+$ such that $k_1+k_2+\cdots+k_l=k$. Suppose that $\mu=(\gamma_1^{n_1}\gamma_2^{n_2}\cdots\gamma_l^{n_l})$ is a primitive partition of $\delta$ and $k_1,k_2,\ldots,k_l\in\mathbb Z^+$ such that $k_1+k_2+\cdots+k_l=k$. If $\tilde\sigma^{(h)}\in\mathscr B_{k_h,\gamma_h}(n_h\gamma_h)$ for all $h$, then $\lambda=\tilde\sigma^{(1)}\oplus\tilde\sigma^{(2)}\oplus\cdots\oplus\tilde\sigma^{(l)}\in\mathscr B_k(\delta)$ and $\mu$ is a refinement of $\lambda$. Thus the number $|S'_{\mu,k_1,k_2,\ldots,k_l}|$ of $\lambda=\tilde\sigma^{(1)}\oplus\tilde\sigma^{(2)}\oplus\cdots\oplus\tilde\sigma^{(l)}$ such that $\tilde\sigma^{(h)}\in\mathscr B_{k_h,\gamma_h}(n_h\gamma_h)$ is $\prod_{h=1}^lp(\mathscr B_{k_h,\gamma_h},n_h\gamma_h)$ and so, the number of $\lambda\in\mathscr B_k(\delta)$ such that $\mu$ is a refinement of $\lambda$ is $\sum_{k_1+k_2+\cdots+k_l=k}(\prod_{h=1}^lp(\mathscr B_{k_h,\gamma_h},n_h\gamma_h))$. Thus we have
\begin{align*}
p(\mathscr A_k,\delta)&=\sum_{\mu\in P_{\delta}}\left(\sum_{k_1+k_2+\cdots+k_l=k}\left(\prod_{h=1}^lp(\mathscr A_{k_h,\gamma_h},n_h\gamma_h)\right)\right)\\
&=\sum_{\mu\in P_{\delta}}\left(\sum_{k_1+k_2+\cdots+k_l=k}\left(\prod_{h=1}^l A_{k_h}(n_h)\right)\right)\\
&=\sum_{\mu\in P_{\delta}}\left(\sum_{k_1+k_2+\cdots+k_l=k}\left(\prod_{h=1}^l B_{k_h}(n_h)\right)\right)\\
&=\sum_{\mu\in P_{\delta}}\left(\sum_{k_1+k_2+\cdots+k_l=k}\left(\prod_{h=1}^lp(\mathscr B_{k_h,\gamma_h},n_h\gamma_h)\right)\right)\\
&=p(\mathscr B_k,\delta)
\end{align*}
where $\mu=(\gamma_1^{n_1}\gamma_2^{n_2}\ldots\gamma_l^{n_l})$.
\end{proof}

Let $S$ and $T$ be subsets of $\mathscr S$. We define $S\oplus T=\{\lambda\oplus\sigma|\lambda\in S~{\rm and}~\sigma\in T\}$.

\begin{ex}
Let $K=\mathbb Q(\sqrt3)$ and $\varepsilon=2+\sqrt3$. Let $\gamma_1=1$ and $\gamma_2=\varepsilon$. Let $\delta=7+7\varepsilon=21+7\sqrt3\in\mathcal O^+_K$ and let $\mu=(1^7\varepsilon^7)$ be a primitive partiton of $\delta$. Let $\mathscr A_{k,\gamma}$, $\mathscr B_{k,\gamma}$, $S_{\mu,k_1,k_2,\ldots,k_l}$ and $S'_{\mu,k_1,k_2,\ldots,k_l}$ be the same as those in the proof of Theorem \ref{pA=pB}. Since
\begin{align*}
S_{\mu,1,2}=&~\mathscr A_{1,1}\oplus \mathscr A_{2,\varepsilon}\\
=&~\{(1^7),(7^1)\}\oplus\{(\varepsilon^2(5\varepsilon)^1),(\varepsilon^4(3\varepsilon)^1),(\varepsilon^1(3\varepsilon)^2)\}\\
=&~\{(1^7\varepsilon^2(5\varepsilon)^1),(1^7\varepsilon^4(3\varepsilon)^1),(1^7\varepsilon^1(3\varepsilon)^2),(7^1\varepsilon^2(5\varepsilon)^1),\\
&~(7^1\varepsilon^4(3\varepsilon)^1),(7^1\varepsilon^1(3\varepsilon)^2)\},\\
S_{\mu,2,1}=&~\mathscr A_{2,1}\oplus \mathscr A_{1,\varepsilon}\\
=&~\{(1^25^1),(1^43^1),(1^13^2)\}\oplus\{(\varepsilon^7),((7\varepsilon)^1))\},\\
S'_{\mu,1,2}=&~\mathscr B_{1,1}\oplus \mathscr B_{2,\varepsilon}\\
=&~\{(3^14^1),(7^1)\}\oplus\{(\varepsilon^1(6\varepsilon)^1),((2\varepsilon)^1(5\varepsilon)^1),(\varepsilon^1(2\varepsilon)^1(4\varepsilon)^1)\}
\end{align*}
and
\begin{align*}
S'_{\mu,2,1}=&~\mathscr B_{2,1}\oplus \mathscr B_{1,\varepsilon}=\{(1^16^1),(2^15^1),(1^12^14^1)\}\oplus\{((3\varepsilon)^1(4\varepsilon)^1),((7\varepsilon)^1))\},
\end{align*}
we have $|S_{\mu,1,2}|=|S'_{\mu,1,2}|=6$ and $|S_{\mu,2,1}|=|S'_{\mu,2,1}|=6$. It is calculated that all $|S_{\eta,k_1,k_2}|$ and $|S'_{\eta,k_1,k_2}|$ for all primitive partition $\eta$ of $\delta$ and $(k_1.k_2)=(1,2),(2,1)$ in the same way, by using computer. As a consequence, we have $p(\mathscr A_3,\delta)=p(\mathscr B_3,\delta)=526$.
\end{ex}

\subsection{The Rogers-Ramanujan Identities over Totally Real Number Fields}
In this section, we generalize the Rogers-Ramanujan Identities. In fact, we generalize Gordon's Theorem. Following two equivalent theorems are well known Gordon's Generalization \cite{gor}.
\begin{thm}[\cite{and11}, Theorem 7.5] \label{gor}
Let $B_{k,i}(n)$ denote the number of partitions of $n$ of the form $\lambda=(b_1,b_2,\ldots,b_s)$ such that $b_1\ge b_2\ge\ldots\ge b_s$, $b_j-b_{j+k-1}\ge2$ and $b_{s-i+1}\ge2$. Let $A_{k,i}(n)$ denote the number of partitions of $n$ into parts $\not\equiv0,\pm i\pmod{2k+1}$. Then $A_{k,i}(n)=B_{k,i}(n)$ for all $n$.
\end{thm}

Note that $B_{k,i}(n)$ is the number of partitions $\lambda$ of $n$ such that for each part $\lambda_i$ of $\lambda$, the number of parts that are equal to $\lambda_i$ or $\lambda_i-1$ is at most $k-1$ and the number of parts that are equal to $1$ is at most $i-1$. Also it is the number of elements of the set 
$\{\lambda=(\lambda_1,\lambda_2,\ldots,\lambda_s)|\lambda_1+\lambda_2+\cdots+\lambda_s=n,\lambda_1\ge\lambda_2\ge\ldots\ge\lambda_s\ge1, \lambda_j-\lambda_{j+k-1}\ge2 ~{\rm and}~ \lambda_{s-i+1}\ge2\}$.

Gordon's Theorem is equivalent to the following identity.

\begin{thm}[\cite{and11}, Theorem 7.8] \label{gor2}
For $1\le i\le k$, $k\ge2$ and $|q|<1$,
\begin{align*}
\sum_{n_1,n_2,\ldots,n_{k-1}\ge0}\frac{q^{N_1^2+N_2^2+\cdots+N_{k-1}^2+N_i+N_{i+1}+\cdots+N_{k-1}}}{(q)_{n_1}(q)_{n_2}\cdots(q)_{n_{k-1}}}=\prod_{\substack{n=1\\n\not\equiv0,\pm i\pmod{2k+1}}}^{\infty}\frac1{1-q^n}
\end{align*}
where $N_j=n_j+n_{j+1}+\cdots+n_{k-1}$.
\end{thm}

From now on, we generalize Gordon's Theorem to all totally real number fields. The starting point is to consider the set of all partitions whose refinement $\sigma$ is a fixed primitive partition.

\begin{lem}\label{sigma}
Suppose positive integers $k$ and $i$ such that $1\le i\le k$ and a primitive partition $\sigma=(\gamma_1^{n_1}\gamma_2^{n_2}\ldots\gamma_l^{n_l})$ of $\delta\in\mathcal O^+$ are given. Let $\mathscr B_{k,i}^{\sigma}(\delta)$ denote the number of partitions $\lambda=\lambda^{(1)}\oplus\lambda^{(2)}\oplus\cdots\oplus\lambda^{(l)}$ of $\delta$ such that $\sigma$ is a refinement of $\lambda$, $\lambda^{(h)}=(b_{h,1}\gamma_h,b_{h,2}\gamma_h,\ldots,b_{h,m_h}\gamma_h)\in{''{\gamma_h\mathbb Z^+}''}$, $b_{h,1}\ge b_{h,2}\ge\ldots\ge b_{h,m_h}\ge1$, $b_{h,j}-b_{h,j+k-1}\ge2$ and $b_{h,m_h-i+1}\ge2$. Let $\mathscr A_{k,i}^{\sigma}(\delta)$ denote the number of partitions $\lambda$ such that $\sigma$ is a refinement of $\lambda$ and each part of $\lambda$ is of scale not congruent to $0, i,-i$ modulo $2k+1$. Then $\mathscr A_{k,i}^{\sigma}(\delta)=\mathscr B_{k,i}^{\sigma}(\delta)$ for all $\sigma\in P_{\delta}$.
\end{lem}
\begin{proof}
For each $\gamma\in\mathcal P$, we denote $\mathscr A'_{k,i,\gamma}$ the set of all partitions $\lambda\in{''{\gamma\mathbb Z^+}''}$ each of whose part is of scale not congruent to $0, i,-i$ modulo $2k+1$. Let $\mathscr B'_{k,i,\gamma}$ denote the set of partitions of the form $\lambda=(b_1\gamma,b_2\gamma,\ldots,b_m\gamma)\in{''{\gamma\mathbb Z^+}''}$ where $b_1\ge b_2\ge\ldots\ge b_m\ge1$, $b_j-b_{j+k-1}\ge2$ for all $j$ and $b_{m-i+1}\ge2$. For all $n\in\mathbb Z^+$ and $\gamma\in\mathcal P$, by Lemma \ref{pgnn}, we have $A_{k,i}(n)=p(\mathscr A_{k,i,1}',n)=p(\mathscr A_{k,i,\gamma}',n\gamma)$ and $B_{k,i}(n)=p(\mathscr B_{k,i,1}',n)=p(\mathscr B_{k,i,\gamma}',n\gamma)$. Here $A_{k,i}(n)$ and $B_{k,i}(n)$ are the same as those in Theorem \ref{gor}. Note that $\mathscr B^{\sigma}_{k,i}(\delta)$ is the number of partitions $\lambda=\lambda^{(1)}\oplus\lambda^{(2)}\oplus\cdots\oplus\lambda^{(l)}$ of $\delta$ such that $\lambda^{(h)}\in\mathscr B'_{k,i,\gamma_h}$ and $\lambda^{(h)}\vdash n_h\gamma_h$ for all $h$. Since the number of $\lambda^{(h)}$ satisfying these conditions is $p(\mathscr B'_{k,i,\gamma_h},n_h\gamma_h)$ for all $h$, we have $\mathscr B^{\sigma}_{k,i}(\delta)=\prod_{h=1}^l p(\mathscr B'_{k,i,\gamma_h},n_h\gamma_h)$. Also $\mathscr A^{\sigma}_{k,i}(\delta)$ is the number of partitions $\lambda$ of $\delta$ such that $\sigma$ is a refinement of $\lambda$ and each part of $\lambda$ is not of scale congruent to $0, i, -i$ modulo $2k+1$. Thus $\lambda=\lambda^{(1)}\oplus\lambda^{(2)}\oplus\cdots\oplus\lambda^{(l)}$ such that $\lambda^{(h)}\in''{\gamma_h\mathbb Z^+}''$ for all $h$. We can see that $\lambda^{(h)}\vdash n_h\gamma_h$ and each part of $\lambda^{(h)}$ is also of scale not congruent to $0, i, -i$ modulo $2k+1$ for all $h$. Thus the number of such $\lambda^{(h)}$ is $p(\mathscr A'_{k,i,\gamma_h},n_h\gamma_h)$ for all $h$, and hence $\mathscr A^{\sigma}_{k,i}(\delta)=\prod_{h=1}^l p(\mathscr A'_{k,i,\gamma_h},n_h\gamma_h)$. By Theorem \ref{gor} and Corollary \ref{psp}, we have 
\begin{align*}
\mathscr A_{k,i}^{\sigma}(\delta)&=\prod_{h=1}^lp(\mathscr A_{k,i,\gamma_h}',n_h\gamma_h)=\prod_{h=1}^lA_{k,i}(n_h)\\
&=\prod_{h=1}^lB_{k,i}(n_h)=\prod_{h=1}^lp(\mathscr B_{k,i,\gamma_h}',n_h\gamma_h)=\mathscr B_{k,i}^{\sigma}(\delta).
\end{align*}
\end{proof}

We can see that $\mathscr B_{k,i}^{\sigma}(\delta)$ is the number of partitions $\lambda=\lambda^{(1)}\oplus\lambda^{(2)}\oplus\cdots\oplus\lambda^{(l)}$ of $\delta$ satisfying the following three conditions for all $h$.
\begin{enumerate}
\item[1.] $\lambda^{(h)}$ is a $\gamma_h$-sectional partition of $n_h\gamma_h$, 
\item[2.] for each part $\lambda_{h,j}$ of $\lambda^{(h)}$ there are at most $k-1$ parts of $\lambda^{(h)}$ whose scale is $s(\lambda_{h,j})$ or $s(\lambda_{h,j})-1$, 
\item[3.] there are at most $i-1$ primitive parts of $\lambda^{(h)}$. 
\end{enumerate}

Let $\mathscr A_{k,i}'$ be the set of all partitions whose each part is of scale not congruent to $0, i,-i$ modulo $2k+1$. Let $\mathscr B_{k,i}'$ be the set of all partitions $\lambda^{(1)}\oplus\lambda^{(2)}\oplus\cdots\oplus\lambda^{(l)}$ such that there are pairwise distinct $\gamma_1,\gamma_2,\ldots,\gamma_l$ such that $\lambda^{(h)}\in\mathscr B_{k,i,\gamma_h}'$ for all $h$. Note that both $\mathscr A_{k,i}'$ and $\mathscr B_{k,i}'$ are closed under $\oplus$ for disjoint sectional partitions.

\begin{thm}\label{mathscrAB}
For all $\delta\in\mathcal O^+$, we have $p(\mathscr A_{k,i}',\delta)=p(\mathscr B_{k,i}',\delta)$.
\end{thm}
\begin{proof}
By Corollary \ref{psp} and Lemma \ref{sigma}, we have
\begin{align*}
p(\mathscr A_{k,i}',\delta)&=\sum_{\sigma\in P_{\delta}}\left(\prod_{h=1}^lp(\mathscr A_{k,i}'\cap{''{\gamma_h\mathbb Z^+}''},n_h\gamma_h)\right)\\
&=\sum_{\sigma\in P_{\delta}}\left(\prod_{h=1}^lp(\mathscr A_{k,i,\gamma_h}',n_h\gamma_h)\right)=\sum_{\sigma\in P_{\delta}}\mathscr A_{k,i}^{\sigma}\\
&=\sum_{\sigma\in P_{\delta}}\mathscr B_{k,i}^{\sigma}=\sum_{\sigma\in P_{\delta}}\left(\prod_{h=1}^lp(\mathscr B_{k,i,\gamma_h}',n_h\gamma_h)\right)\\
&=\sum_{\sigma\in P_{\delta}}\left(\prod_{h=1}^lp(\mathscr B_{k,i}'\cap{''{\gamma_h\mathbb Z^+}''},n_h\gamma_h)\right)=p(\mathscr B_{k,i}',\delta)
\end{align*}
where $\sigma=(\gamma_1^{n_1}\gamma_2^{n_2}\ldots\gamma_l^{n_l})\in P_{\delta}$. 
\end{proof}

The following theorem is a formal $q$-sum version of Theorem \ref{mathscrAB}. 

\begin{thm}\label{rrgf}
For $k\ge2$ and $1\le i\le k$, 
\begin{align*}
&\prod_{\gamma\in\mathcal P}\left(\sum_{n_1,n_2,\ldots,n_{k-1}\ge0}\frac{q^{(N_1^2+N_2^2+\cdots+N_{k-1}^2+N_i+N_{i+1}+\cdots+N_{k-1})\gamma}}{(q^{\gamma})_{n_1}(q^{\gamma})_{n_2}\cdots(q^{\gamma})_{n_{k-1}}}\right)\\
=&\prod_{\gamma\in\mathcal  P}\left(\prod_{n\not\equiv0,\pm i\pmod{2k+1}}\frac1{1-q^{n\gamma}}\right)=\prod_{\substack{\delta\in\mathcal O^+\\s(\delta)\not\equiv0,\pm i\pmod{2k+1}}}\frac1{1-q^{\delta}}
\end{align*}
where $N_j=n_j+n_{j+1}+\cdots+n_{k-1}$.
\end{thm}
\begin{proof}
By Lemma \ref{fg1q} and Theorem \ref{gor2}, we have
\begin{align*}
&\prod_{\gamma\in\mathcal  P}\left(\sum_{n_1,n_2,\ldots,n_{k-1}\ge0}\frac{q^{(N_1^2+N_2^2+\cdots+N_{k-1}^2+N_i+N_{i+1}+\cdots+N_{k-1})\gamma}}{(q^{\gamma})_{n_1}(q^{\gamma})_{n_2}\cdots(q^{\gamma})_{n_{k-1}}}\right)\\
=&\prod_{\gamma\in\mathcal  P}\left(\prod_{n\not\equiv0,\pm i\pmod{2k+1}}\frac1{1-q^{n\gamma}}\right).
\end{align*}
Since 
\begin{align*}
&\{\delta|\delta\in\mathcal O^+~{\rm and}~s(\delta)\not\equiv0,\pm i\pmod{2k+1}\}\\
=&\{n\gamma|\gamma\in\mathcal P~{\rm and}~n\not\equiv0,\pm i\pmod{2k+1}\},
\end{align*}
we have 
\begin{align*}
\prod_{\gamma\in\mathcal P}\left(\prod_{n\not\equiv0,\pm i\pmod{2k+1}}\frac1{1-q^{n\gamma}}\right)=\prod_{\substack{\delta\in\mathcal O^+\\s(\delta)\not\equiv0,\pm i\pmod{2k+1}}}\frac1{1-q^{\delta}}.
\end{align*}
\end{proof}

From Theorems \ref{mathscrAB} and \ref{rrgf}, we have 
\begin{align*}
&\sum_{\delta\in\mathcal O^+}p(\mathscr B_{k,i}',\delta)q^{\delta}=\sum_{\delta\in\mathcal O^+}p(\mathscr A_{k,i}',\delta)q^{\delta}=\prod_{\substack{\delta\in\mathcal O^+\\s(\delta)\not\equiv0,\pm i\pmod{2k+1}}}\frac1{1-q^{\delta}}\\
=&\prod_{\gamma\in\mathcal  P}\left(\sum_{n_1,n_2,\ldots,n_{k-1}\ge0}\frac{q^{(N_1^2+N_2^2+\cdots+N_{k-1}^2+N_i+N_{i+1}+\cdots+N_{k-1})\gamma}}{(q^{\gamma})_{n_1}(q^{\gamma})_{n_2}\cdots(q^{\gamma})_{n_{k-1}}}\right)
\end{align*}
where $N_j=n_j+n_{j+1}+\cdots+n_{k-1}$.

Theorem \ref{mathscrAB} is also equivalent to the following theorem, which gives an intuitive interpretation.

\begin{thm}\label{gortrnf}
Suppose $k\ge2$ and $1\le i\le k$. Then for all $\delta\in\mathcal O^+$ the partition $\lambda$ of $\delta$ such that for each part $\lambda_i$ of $\lambda$ there are at most $k-1$ parts $\lambda_j$ such that $t(\lambda_j)=t(\lambda_i)$ and $s(\lambda_i)-s(\lambda_j)$ is $0$ or $1$, and for each primitive element $\gamma$ there are at most $i-1$ parts $\lambda_j$ such that $\lambda_j=\gamma$, is equinumerous with the partitions all of whose parts are of scale not congruent to $0,i,-i$ modulo $2k+1$.
\end{thm}

The Rogers-Ramanujan Identities over a totally real number field $K$ is obtained by Theorem \ref{gortrnf} as special cases. The first identity is the case $k=i=2$, and the second one is the case $k=2$ and $i=1$.

\begin{cor}[The First Rogers-Ramanujan Identity]
The partitions of $\delta\in\mathcal O^+$ in which the scale of each part is congruent to $1$ or $4$ modulo $5$ are equinumerous with the partitions of $\delta$ in which the difference between the scales of any two parts in the same section is at least $2$.
\end{cor}

\begin{cor}[The Second Rogers-Ramanujan Identity]
The partitions of $\delta\in\mathcal O^+$ in which the scale of each part is congruent to $2$ or $3$ modulo $5$ are equinumerous with the partitions of $\delta$ in which there is no primitive part and the difference between scales of any two parts in the same section is at least $2$.
\end{cor}

The following corollary is also consequences of Theorem \ref{rrgf} for $k=i=2$ and $k=2$, $i=1$.

\begin{cor}We have
\begin{align*}
\prod_{\gamma\in\mathcal P}\left(\sum_{n=0}^{\infty}\frac{q^{n^2\gamma}}{(q^{\gamma})_n}\right)&=\prod_{\gamma\in\mathcal P}\left(\prod_{n=0}^{\infty}\frac1{(1-q^{(5n+1)\gamma})(1-q^{(5n+4)\gamma})}\right)\\
&=\prod_{\substack{\delta\in\mathcal O^+\\s(\delta)\equiv1,4\pmod5}}\frac1{1-q^{\delta}}
\end{align*}
and
\begin{align*}
\prod_{\gamma\in\mathcal P}\left(\sum_{n=0}^{\infty}\frac{q^{n(n+1)\gamma}}{(q^{\gamma})_n}\right)&=\prod_{\gamma\in\mathcal P}\left(\prod_{n=0}^{\infty}\frac1{(1-q^{(5n+2)\gamma})(1-q^{(5n+3)\gamma})}\right)\\
&=\prod_{\substack{\delta\in\mathcal O^+\\s(\delta)\equiv2,3\pmod5}}\frac1{1-q^{\delta}}.
\end{align*}
\end{cor}

\subsection{Another Version of the Generalized Rogers-Ramanujan Identities}

Let $K$ be a totally real number field with extension degree $n$. Suppose that there is a prime ideal $\mathfrak p$ of $\mathcal O=\mathcal O_K$ such that $5\in\mathfrak p$ and the residue class field $\mathcal O/\mathfrak p$ is isomorphic to $\mathbb Z/5\mathbb Z$. Then every $\delta\in\mathcal O$ is congruent to $i$ modulo $\mathfrak p$ for some $i\in\{0,1,2,3,4\}$, i.e., there is $0\le i\le4$ such that $\delta-i\in\mathfrak p$. Note that $\delta\equiv1,4\pmod{\mathfrak p}$ if and only if either $t(\delta)\equiv1,4\pmod{\mathfrak p}$ and $s(\delta)\equiv1,4\pmod5$, or $t(\delta)\equiv2,3\pmod{\mathfrak p}$ and $s(\delta)\equiv2,3\pmod5$ where $t(\delta)$ and $s(\delta)$ are the primitive factor and the scale of $\delta$ respectively. Also $\delta\equiv2,3\pmod{\mathfrak p}$ if and only if either $t(\delta)\equiv1,4\pmod{\mathfrak p}$ and $s(\delta)\equiv2,3\pmod5$, or $t(\delta)\equiv2,3\pmod{\mathfrak p}$ and $s(\delta)\equiv1,4\pmod5$. Let $\lambda=(\alpha_1^{n_1}\alpha_2^{n_2}\cdots\alpha_k^{n_k}\beta_1^{m_1}\beta_2^{m_2}\cdots\beta_l^{m_l})$ be a partition of $\delta\in\mathcal O^+$ with $t(\alpha_i)\equiv1,4\pmod{\mathfrak p}$ and $t(\beta_j)\equiv2,3\pmod{\mathfrak p}$. Then $\alpha_i\equiv\beta_j\equiv1,4\pmod{\mathfrak p}$ if and only if $s(\alpha_i)\equiv1,4\pmod5$ and $s(\beta_j)\equiv2,3\pmod5$. Also $\alpha_i\equiv\beta_j\equiv2,3\pmod{\mathfrak p}$ if and only if $s(\alpha_i)\equiv2,3\pmod5$ and $s(\beta_j)\equiv1,4\pmod5$. By Lemma \ref{fg1q}, for all $\mathcal P'\subset\mathcal P$, we have
\begin{align*}
\prod_{\gamma\in\mathcal P'}\left(\prod_{n=0}^{\infty}\frac1{(1-q^{(5n+1)\gamma})(1-q^{(5n+4)\gamma})}\right)=\prod_{\gamma\in\mathcal P'}\left(\sum_{n=0}^{\infty}\frac{q^{n^2\gamma}}{(q^{\gamma})_n}\right)
\end{align*}
and
\begin{align*}
\prod_{\gamma\in\mathcal P'}\left(\prod_{n=0}^{\infty}\frac1{(1-q^{(5n+2)\gamma})(1-q^{(5n+3)\gamma})}\right)=\prod_{\gamma\in\mathcal P'}\left(\sum_{n=0}^{\infty}\frac{q^{n(n+1)\gamma}}{(q^{\gamma})_n}\right),
\end{align*}
and hence the following two corollaries follow.

\begin{cor}
Let $X=\{\gamma\in\mathcal P|\gamma\equiv1,4\pmod{\mathfrak p}\}$ and $Y=\{\gamma\in\mathcal P|\gamma\equiv2,3\pmod{\mathfrak p}\}$. We have
\begin{align*}
\prod_{\substack{\delta\in\mathcal O^+\\\delta\equiv1,4\pmod{\mathfrak p}}}\frac1{1-q^{\delta}}=&\prod_{\gamma\in X}\left(\prod_{n=0}^{\infty}\frac1{(1-q^{(5n+1)\gamma})(1-q^{(5n+4)\gamma})}\right)\\
\times&\prod_{\gamma\in Y}\left(\prod_{n=0}^{\infty}\frac1{(1-q^{(5n+2)\gamma})(1-q^{(5n+3)\gamma})}\right)\\
=&\prod_{\gamma\in X}\left(\sum_{n=0}^{\infty}\frac{q^{n^2\gamma}}{(q^{\gamma})_n}\right)\prod_{\gamma\in Y}\left(\sum_{n=0}^{\infty}\frac{q^{n(n+1)\gamma}}{(q^{\gamma})_n}\right)
\end{align*}
and
\begin{align*}
\prod_{\substack{\delta\in\mathcal O^+\\\delta\equiv2,3\pmod{\mathfrak p}}}\frac1{1-q^{\delta}}=&\prod_{\gamma\in X}\left(\prod_{n=0}^{\infty}\frac1{(1-q^{(5n+2)\gamma})(1-q^{(5n+3)\gamma})}\right)\\
\times&\prod_{\gamma\in Y}\left(\prod_{n=0}^{\infty}\frac1{(1-q^{(5n+1)\gamma})(1-q^{(5n+4)\gamma})}\right)\\
=&\prod_{\gamma\in X}\left(\sum_{n=0}^{\infty}\frac{q^{n(n+1)\gamma}}{(q^{\gamma})_n}\right)\prod_{\gamma\in Y}\left(\sum_{n=0}^{\infty}\frac{q^{n^2\gamma}}{(q^{\gamma})_n}\right).
\end{align*}
\end{cor}

\begin{cor}
1. The partitions of $\delta$ in which each part is congruent to $1$ or $4$ modulo $\mathfrak p$ is equinumerous with the partitions in which
\begin{enumerate}
\item[(i)] the primitive factor of each part is not congruent to $0$ modulo $\mathfrak p$,
\item[(ii)] the difference between scales of any two parts in the same section is at least $2$,
\item[(iii)] there is no primitive part congruent to $2$ or $3$ modulo $\mathfrak p$. 
\end{enumerate}
2. The partitions of $\delta$ in which each part is congruent to $2$ or $3$ modulo $\mathfrak p$ is equinumerous with the partitions in which
\begin{enumerate}
\item[(i)] the primitive factor of each part is not congruent to $0$ modulo $\mathfrak p$,
\item[(ii)] the difference between scales of any two parts in the same section is at least $2$,
\item[(iii)] there is no primitive part congruent to $1$ or $4$ modulo $\mathfrak p$. 
\end{enumerate}

\end{cor}

\bibliography{aomsample}
\bibliographystyle{aomplain}

\end{document}